\documentclass[abstract,11pt]{scrartcl}
\usepackage{graphicx} 
\usepackage[english]{babel}
\usepackage{csquotes}
\usepackage{microtype}
\usepackage{mathtools}
\usepackage{amssymb, dsfont}
\usepackage{amsthm}
\usepackage{hyperref}
\usepackage{algorithm}
\usepackage{algpseudocodex}

\hypersetup{ 
    colorlinks=true,
    linkcolor=blue,
    urlcolor=cyan,
    citecolor=black 
}

\newtheorem{lma}{Lemma}
\newtheorem*{prp*}{Proposition}

\DeclareMathOperator{\Xc}{\mathcal{X}}
\DeclareMathOperator{\Yc}{\mathcal{Y}}

\DeclareMathOperator{\R}{\mathbb{R}}

\title{Genetic Column Generation for Computing Lower Bounds for Adversarial Classification}
\date{June 11, 2024}
\author{Maximilian Penka\footnote{ Department of Mathematics, Technische Universit\"{a}t M\"{u}nchen, Germany\\Email: \texttt{penka@ma.tum.de}}}

\begin{document}
\maketitle

\begin{abstract}
\noindent Recent theoretical results on adversarial multi-class classification showed a similarity to the multi-marginal formulation of Wasserstein-barycenter in optimal transport. Unfortunately, both problems suffer from the curse of dimension, making it hard to exploit the nice linear program structure of the problems for numerical calculations. We investigate how ideas from Genetic Column Generation for multi-marginal optimal transport can be used to overcome the curse of dimension in computing the minimal adversarial risk in multi-class classification.
\end{abstract}

\section{Introduction}
Multi-class classification is a standard task in data science.
While it is easy to train a classifier with (almost) vanishing risk on the training data, it is also known that those methods are often not very robust to small perturbations of the data points 
\cite{goodfellow2014explaining, nguyen2015deep}.

Hence, the challenge has shifted to finding robust classifiers. 
A well-established approach is adversarial training, where an attacker is allowed to slightly perturb the data distribution in order to maximize the risk of the classification mimicking a two-player game 
\cite{goodfellow2014explaining, madry2017towards, bai2021recent, zhao2022adversarial}.

The classical ansatz is to allow the attacker to maximize the loss by perturbing a data point within an $\varepsilon$-ball -- called budget -- with respect to the metric of the feature space. From an optimal transport perspective, this is equivalent to perturbing the empirical measure induced by the data set within a ball in the Wasserstein space $W_\infty$ with radius $\varepsilon$ in order to maximize the empirical risk of the classifier.

While the initial motivation for that problem was to find a robust training strategy, the ansatz can be generalized to a distributional setting to study the problem independent of the training procedure \cite{sinha2017certifying,esfahani2018data}. A fundamental theoretical problem in adversarial classification addressed in this work is the following. Is my data set sufficient to train a robust classifier? More precisely, that is the minimal risk any classifier can achieve, given a data set and an adversarial budget.

In a recent work of \cite{trillos2023multimarginal}, a reformulation of this lower bound was found that can be seen as relaxation in linear program (LP) form and, in its structure, is related to the barycenter problem in optimal transport \cite{agueh2011barycenters}. Unfortunately, the number of unknowns in that problem scales polynomially in the number of data points and even exponentially in the number of classes. However, from the LP structure, it is known that the problem admits an extremely sparse solution.

In a follow-up work \cite{trillos2024optimal}, the authors provided a numerical approach using truncation and sub-sampling and argued that for data sets with little overlap of many classes, that provides a good approximation.

In this paper those limitations are addressed. First, for a data set of fixed size $N$ (which will correspond to the number of constraints in the linear program), a high number of classes corresponds to a low number of data points per class. Further sub-sampling would then jeopardize the approximation validity of the empirical measure.

Coming from multi-marginal optimal transport (MMOT), we will choose a different approach. The recently introduced Genetic Column Generation Algorithm (GenCol) is an efficient routine to solve multi-marginal problems by generating candidate configurations in a genetic fashion and maintaining a sparse set of configurations.
For MMOT problems arising in quantum physics, as well as for Wasserstein-barycenter and -splines, the algorithm showed impressive performance. \cite{friesecke2022genetic,friesecke2023gencol}.
By invoking the ideas from genetic column generation, we will tackle the second limitation. The algorithms presented here do not exclude any configurations from the outset. The reduction of the problem size is dynamically maintained in a genetic fashion.

The paper is structured as follows: In Section \ref{sec:AdvRisk}, we first set up the mathematical background, mainly following \cite{trillos2023multimarginal}. The second part introduces the linear program and translates the problem into the language for GenCol.

The next section discusses the genetic search rules, which need to be modified for the problem. In contrast to a pure multi-marginal transport problem, the new problem features configurations of different lengths.

In addition, we will explore the effects of a different penalty term by replacing the classical adversarial budget -- related to a bound on the Wasserstein-$\infty$ deviation from the training data -- with a Wasserstein-2 penalty. 
That enables us to use duality as a powerful critic to accelerate the genetic search for new configurations.

The last part demonstrates the application of the algorithm to synthetic data of 10 classes with a huge overlap in classes and a subset of 30 classes of the CIFAR-100 data set \cite{krizhevsky2009learning}.
In both examples, the data set has the property that truncation would lead to a significant underestimation of the adversarial risk.

\section{Adversarial risk for classification problems}
\label{sec:AdvRisk}
The starting point for the following considerations is a generalization of the adversarial classification problem following the work of \cite{trillos2023multimarginal}.
As mentioned in the introduction, we take a distributional perspective. We assume the data set of interest is distributed according to a probability measure $\mu$ on the Cartesian product of a Polish feature space $\Xc$ and a finite label space $\Yc = \{1,\dots,K\}$, the latter equipped with the discrete topology.
A realization is hence a pair $(x,i)$ of feature $x\in\Xc$ and label $i \in \Yc$.
We then consider probabilistic classifiers $f:\Xc \to \{p \in [0,1]^K : |p|_1 = 1\}$. 
The quantity $f_i(x)$ is simply the probability that point $x \in \Xc$ belongs to class $i \in \Yc$.
The set of classifiers is the set of (Borel-)measurable functions like that and will be denoted by $\mathcal{F}$.

The metric of interest is the risk of a classifier with respect to a loss of 0-1 type:
\[R(f,\mu) := \int_{\Xc\times \Yc} 1-f_y(x)\ d\mu(x,y). \]
Since the label set $\Yc$ is discrete, we can slice the measure $\mu$ along the classes $i\in \Yc$ into $\mu_i := \mu(\cdot \times \{i\})$.
The $\mu_i$ are positive measures on $\Xc$, but not necessarily probability measures because, in contrast to conditional probabilities, they are not normalized. 
That yields the decomposition of the risk
\[R(f,\mu) = 1- \sum_{i=1}^N \int_{\Xc}f_i(x)\ d\mu_i(x).\]
The classical (distributional) adversarial risk for an adversarial budget $\varepsilon > 0$ is
\begin{equation} \label{eq:probadvrisk}
    \sup_{\tilde \mu}\,\Big\{R(f,\tilde \mu) ~\Big\vert~ W_\infty(\tilde \mu_i,\mu_i) \leq \varepsilon \, \forall i \in \Yc \Big\},
\end{equation}
where $W_\infty$ is the Wasserstein-$\infty$ distance of $\mu_i$ and $\tilde \mu_i$. 
The definition of distributional adversarial risk can easily be generalized by replacing the $W_\infty$-ball with a general penalty term on $\mu$ and $\tilde \mu$:
\[ \sup_{\tilde \mu} \ R(f,\tilde \mu) - C(\mu,\tilde \mu). \]
The domain of $\tilde \mu$ and the formulation of $C$ is a modeling choice, and it is non-trivial that the problem is well-posed \cite{trillos2023existence, trillos2023multimarginal}. 
For now, we will continue with this general setup and specify concrete functions and domains later.

As explained in the introduction, we are interested in the lower bound for the adversarial risk for \emph{any} classifier $f$ in the set of all probabilistic classifiers $\mathcal{F}$. We study therefore the saddle point problem 
\begin{equation} \label{eq:relaxedform}
   \adjustlimits  \inf_{f \in \mathcal{F}} \sup_{\tilde \mu } \ R(f,\tilde \mu) - C(\mu,\tilde \mu).
\end{equation}
The idea from \cite{trillos2023multimarginal} is to study penalty terms related to optimal transport problems. Because one wants the adversarial attack to act only on the state space $\Xc$, one again decomposes the measures $\mu$ and $\tilde \mu$ in $\mu_1,\dots,\mu_K$, resp. $\tilde\mu_1,\dots,\tilde\mu_K$ along the classes $i \in \Yc$ to define
\begin{equation}\label{eq:defC}
    C(\mu,\tilde\mu) := \sum_{i=1}^K \inf\left\{\int_{\Xc} c(x,y)\,d\gamma(x,y) ~\middle\vert~ \gamma \in \Pi(\mu_i,\tilde \mu_i)\right\},
\end{equation}
where $\Pi(\mu_i,\tilde \mu_i)$ is the set of transport plans for the marginals $\mu_i$ and $\tilde \mu_i$. The cost function $c: \Xc \times \Xc \to [0,\infty]$ has to be lower semi-continuous and satisfy $c(x,x) = 0$. This definition ensures $\mu_i(\Xc) = \tilde \mu_i(\Xc)$ for all $i = 1,\dots,K$ because otherwise, the set of transport plans is empty and by convention the infimum is $+\infty$.

The authors then show that problem \eqref{eq:relaxedform} is equivalent to a problem which has a structure similar to a multi-marginal optimal transport problem:
\begin{subequations}
\label{eq:MMOT-like}
\begin{align} 
    \inf_{\{\gamma_A\}}\quad &\smashoperator[l]{\sum_{A \in S_K}} \int_{\Xc^{|A|}} c_A + 1 \ d\,\gamma_A \\ 
    \text{subject to } &\smashoperator[l]{\sum_{{A \in S_K(i)}}} (e_i)_\sharp\gamma_A = \mu_i, \textrm{ for all } i \in Y.
\end{align}
\end{subequations}
The set $S_K$ is the power set of $\Yc = \{1,\dots,K\}$ except for the empty set, the set $S_K(i)$ is its subset of all sets containing $i \in \Yc$.
$\gamma_A$ are positive measures on the product space $\Xc^{|A|}$, and $(e_i)_\sharp\gamma_A$ the push-forward under the evaluation function on the $i$-th marginal, $i\in A$. 
Note that only $\sum_{A \in S_K} \gamma_A$ is a probability measure.
$c_A$ is the cost function assigning a cost to each configuration $w \in \Xc^{|A|}$ and is directly derived from the cost function $c$ in \eqref{eq:defC}:
\begin{equation}
    c_A(\omega) := \inf_{\bar x\in\Xc} \sum_{x_i \in \omega} c(x_i,\bar x)
\end{equation}
The relation between the two problems is $\eqref{eq:relaxedform} = 1 - \eqref{eq:MMOT-like}$.

To recover the adversarial attacks described in \eqref{eq:probadvrisk} from this generalized formulation, choose $c_A(w)$ to be 0 if the radius of the smallest enclosing ball with respect to the metric $d$ on $\Xc$ of the configuration $w \in \Xc^{|A|}$ is less than the budget $\varepsilon$, and $+\infty$ else. 

Surprisingly, this formulation is much better tractable because in its discrete version -- $\mu$ approximated by the empirical measures $\mu^N := \frac{1}{N}\sum_i \delta_{(x_i,y_i)}$ -- it becomes a linear program. That shall be the starting point for all the following considerations.

\subsection{Discretization and linear program formulation}

The linear program's variables $\gamma(r)$ shall be indexed by configurations $r \subset \{1,\dots,N\}$. 
A \emph{feasible configuration} $r = \{r_1,\dots,r_m\}$ must have pairwise distinct classes $y_{r_j} \neq y_{r_l}$. 
That implies $m \leq K$. Those labels $\{y_{r_1},\dots,y_{r_m}\}$ correspond to the sets $A$ in the general formulation above and will be called the \emph{classes of the configuration} $r$.

For the classical adversarial problem \eqref{eq:probadvrisk}, the value of $c(r)$ simply depends on the radius of the smallest enclosing ball of $\{x_{r_1},...,x_{r_m}\}$ being less than the adversarial budget $\varepsilon$.
Hence, we define the radius of a configuration $r$
\[\operatorname{radius}(r) := \inf\{ \delta > 0 : \exists x \in \Xc \text{ with } \{x_{r_1},\dots,x_{r_m}\} \subset B_\delta(x) \},\]
or equivalently
\[ \operatorname{radius}(r) := \inf\{\delta > 0 : \min_{x \in \Xc} \max_{y \in \{x_{r_1},...,x_{r_m}\}} d(x,y) \leq \delta\}. \]
Compared to the general notation above, we get the mapping $r \mapsto (x_{r_1},\dots,x_{r_m}) \mathrel{\widehat{=}} w$, and $c$ does not depend on $A$ because the smallest enclosing ball can be defined independent of the length of the configuration $r$.

\noindent The cost coefficients are defined
\[
c(r) = \begin{cases} 1 &\operatorname{radius}(r) \leq \varepsilon\\ +\infty &else. \end{cases}
\]
With this setup, the objective becomes the linear program
\begin{subequations}
	\label{eq:LP}
\begin{align}
	\operatorname*{minimize}\ &\sum_{r \text{ feasible }}  c(r)\gamma(r) ,
	\shortintertext{subject to the constraints}
	&\sum_{r \text{ feasible }} \mathds{1}_{\{i \in r\}} \gamma(r) = \mu(\{(x_i,y_i)\}) = \frac{1}{N} \text{ for all } i = 1,...,N.
\end{align}
\end{subequations}
We will refer to the optimal value of the minimization problem as the optimal cost. The challenge shifts to the problem of finding all feasible configurations, denoted by $\overline{\Omega}$. 
In general, this is impossible and should not be done because the problem suffers from the curse of dimensions; its number is $\sum_{A\in S_K} \prod_{i \in A} n_i$. For example, for $K = 10$ classes with $n_i = 100$ data points per class, it is $\sum_{k = 1}^{10} {10\choose k} 100^k > 10^{20}$.\\

In standard LP form, $\min c^{T}\gamma$ s.t.\ $A\gamma=\mu, \gamma \geq 0$, all cost coefficients are either $1$ or $+\infty$.
The constraint matrix $A$ has $N$ rows, one per data point.
Its columns correspond to the configurations $r$: 1 if the data point is in the configuration and 0 everywhere else. Therefore, the matrix is very sparse.

The crucial observation for a promising computational approach is that a sparse optimal solution $\gamma^\star$ exists.
\begin{prp*}
	Problem \eqref{eq:LP} admits an optimizer $\gamma^*$ with  $|\operatorname{spt}\gamma^\star| \leq N$.
\end{prp*}
The support size is independent of the number of classes $K$! 
This fact follows from standard theory:
A linear program in standard form admits a basic solution. The number of variables in the basis is, at most, the rank of the constraint matrix, which is, in that case, less or equal to the number of rows ($N$).

In a similar problem for multi-marginal optimal transport, this observation led to the development of a new algorithm called \emph{Genetic Column Generation} (GenCol) \cite{friesecke2022genetic,friesecke2023gencol}.

The next step is to exclude all configurations with cost $+\infty$.


\section{Search for configurations}
The first idea of column generation is to start from a subset $\Omega \subset \overline{\Omega}$ feasible to solve the \emph{reduced linear program}
\begin{equation} \tag{RP} \label{eq:RP}
\begin{split}
    \operatorname*{minimize}_{\gamma:\Omega \to [0,1]}\quad &\sum_{r \in \Omega}  c(r)\gamma(r) \\
    \text{subject to } &\sum_{r \in \Omega} \mathds{1}_{\{i \in r\}} \gamma(r) = 1/N \text{ for all } i = 1,...,N.
\end{split}
\end{equation}
Feasibility here means that the set of solutions of the restricted LP is non-empty, i.e. there exists a solution to the problem.
Next, one repeatedly adds new configurations to $\Omega$ and resolves the reduced LP.

By adding variables while not adding any constraints, the optimal cost of each reduced problem is monotonically decreasing. Each optimal solution of a reduced problem is an admissible point for the next, enlarged, reduced problem and also for the full problem.

For the classical problem \eqref{eq:probadvrisk}, one makes two observations:
\begin{lma} \label{lma:configs}
    Let $\varepsilon > 0$ and
    \[c(r) := \begin{cases*}
        1 & if there exists $x \in \Xc, \delta \leq \varepsilon$ such that $r \subset B_\delta(x)$ \\
        +\infty & else.
    \end{cases*} \] 
    For the reduced problem \eqref{eq:RP} it holds true that:
\begin{enumerate}
    \item Each singleton configuration $r$ has cost $c(r) = 1$
    \item A configuration $\{r_1,..,r_m\}$ can only have finite cost if each configuration with any of its $r_i$ left out has finite cost.
\end{enumerate}
\end{lma}
\begin{proof}
    1. The ball with radius 0 centered at $x_r$ is just the point. 2. A configuration $r = \{r_1,..,r_m\}$ has finite cost if and only if it fits in a ball with radius $\delta \leq \varepsilon$. But each subset of $r$ fits in the same ball, and therefore the radius of its smallest enclosing ball is less or equal to $\delta$.
\end{proof}

That gives rise to starting the search with all singleton configurations, which is trivially a feasible solution with possibly non-optimal but finite cost. This corresponds to the case that no data point is attacked and $\mu = \tilde \mu$.

By the relation of \eqref{eq:relaxedform} and \eqref{eq:MMOT-like}, the 1 minus the optimal cost of the reduced problem is always a lower bound for the minimal adversarial risk.

Nevertheless, we will briefly start by describing an efficient exhaustive search procedure applicable to small budgets.
\subsection{Search rule 1: Exhaustive search}
Starting from all singleton configurations, iteratively try adding points from foreign classes to a configuration. By Lemma \ref{lma:configs} (ii), one will find all configurations in the full set $\Bar{\Omega}$ with finite cost.

\begin{algorithm}
\caption{Exhaustive search, classical problem \eqref{eq:probadvrisk}}
\label{alg:exhaust}
\begin{algorithmic}[1]
\State $\Omega_1 = \{\{i\} \colon i \in N\}$
\For{$ k = 2,\dots,K$}
    \State $\Omega_k = \emptyset$
    \For{each $r = \{r_1,...,r_{k-1}\} \in \Omega_{k-1}$}
        \For{i = 1,...,N}
            \If{$y_i \notin \{y_{r_1},\dots,y_{r_{k-1}}\}$ and $\operatorname{rad}(\{r_1,\dots,r_{k-1},r_i\}) \leq \varepsilon$}
                    \State Add $\{r_1,\dots,r_{k-1},r_i\}$ to $\Omega_k$
                
            \EndIf
        \EndFor
    \EndFor
\EndFor
\State\Return $\Omega = \bigcup_k \Omega_k$
\end{algorithmic}
\end{algorithm}

For small data sets or budgets so small that there are only a few configurations with finite cost, that is a valid strategy, and an efficient implementation yields the desired result. An efficient implementation includes not trying any configuration twice, splitting the feasible set of configurations of length $k$ in batches, and searching configurations of length $k+1$ in parallel.

However, the number of feasible configurations can scale as bad as $\sum_{A\in S_K} \prod_{i \in A} n_i$, yielding too many configurations to efficiently solve the resulting linear program.

\subsection{Search rule 2: Genetic search}
The following is motivated by the GenCol Algorithm for MMOT problems \cite{friesecke2022genetic,friesecke2023gencol}.
Again, try generating new configurations for $\Omega$ starting from all singleton configurations.
After a fixed amount of configurations were tried, resolve the reduced problem \eqref{eq:RP} with the extended set $\Omega$ to find an extremal solution $\gamma^\star$ with $|\operatorname{spt}\gamma^\star| \leq N$.
Use the current optimal solution of $\eqref{eq:RP}$ to generate new configurations using a genetic search rule and repeat the procedure until no new configurations are found.
The core idea is to consider only \emph{active} configurations, i.e. configurations in the support of $\gamma^\star$, as parents. Due to the fixed number of constraints in the linear program, the size of the support of $\gamma^\star$ is bounded by $N$, limiting the complexity in each search step.
We summarize that in Algorithm \ref{alg:gensearch}.

\begin{algorithm}[!ht]
\caption{Genetic search, classical problem \eqref{eq:probadvrisk}}
\label{alg:gensearch}
\begin{algorithmic}[1]
\State $\Omega = \{\{i\} \colon i \in N\}$
\While{new configurations are found}
    \State Solve \eqref{eq:RP}
    \State $\gamma^\star \leftarrow$ primal solution to \eqref{eq:RP}
    \For{$s$ samples}
        \State Draw parent $r \in \operatorname{spt}(\gamma^\star)$
        \State Generate offspring from the parent
        \If{$\operatorname{radius}(\text{offspring}) \leq \varepsilon$}
            \State Add offspring: $\Omega \leftarrow \Omega \cup \textrm{Offspring}$
        \EndIf
    \EndFor
\EndWhile
\State Solve \eqref{eq:RP}
\end{algorithmic}
\end{algorithm}

To generate offspring, we consider 3 proposal rules:

Rule 1: The first rule is in the philosophy of the exhaustive search; starting from singleton configurations, we add a data point from a foreign class to a parent. If the cost is finite, the offspring are proposed.

Rule 2: The second rule allows points to switch configurations. We pick a parent configuration and exchange one of its entries with a new point from the data set under the restriction that the new point is from a foreign class.

Rule 3: The third rule we considered is the ability of points in a configuration to die. Those offspring always have finite cost by Lemma \ref{lma:configs} and are therefore always accepted.

Note that in an implementation of those routines, one has to check if the configurations are already contained in the set of configurations $\Omega$, indicated in the pseudocode by the union "$\cup$".

By increasing the number of variables while not changing any constraint, the sequence of objective values of the reduced problem is monotonically decreasing.
By the finite number of variables in the full problem, this sequence will eventually converge to a stationary point.
The drawback is that it might not find the global optimizer for the problem. In fact, one expects this routine to have a fast initial decay, as at the beginning, any feasible configuration yields a gain. But it slows down as it gets rare to find improving configurations. The convergence speed observed in numerical experiments is demonstrated in section \ref{sec:experiments}.

The main difference to the exhaustive search is that instead of searching for \emph{all} feasible configurations, the routine can quickly advance to long configurations and then improve from exchanging points.

At this point, all feasible configurations with finite cost are always added. In column generation, in contrast, a critic based on the dual optimal solution decides whether a new configuration is added.

\subsubsection*{Remarks:}
1. All search rules have in common that calculating the smallest enclosing ball must be fast. The calculation is quite easy for the metric induced by the $L^\infty$-norm. One simply evaluates the maximal distance in each coordinate of all features $x_{r_1},\dots,x_{r_k}$.\\[\baselineskip]
\noindent 2. For the metric induced by the $L^2$-norm, calculating the radius of the smallest enclosing ball is a delicate problem. Providentially, there exist theory and efficient implementations of those algorithms \cite{fischer2003smallest}.\\[\baselineskip]
\noindent 3. It might be interesting to hybridize both the genetic and the exhaustive search.

\section{Adversarial attacks with a $W_2$ penalty}
Following ideas from \cite{trillos2023multimarginal}, instead of penalizing the distributional adversarial attack via $W_{\infty}(\tilde \mu ,\mu) < \varepsilon$, one can think of a penalization in the relaxed formulation \eqref{eq:relaxedform} with respect to the $W_2$ distance
\begin{equation} \label{eq:w2penalized}
   \adjustlimits \inf_{f \in \mathcal{F}}\sup_{\tilde \mu} \left\{{R(f,\tilde \mu) - \frac{1}{\tau^2} \sum_{i=1}^K W^2_2(\tilde\mu_i,\mu_i)}\right\}.
\end{equation}
The regularization parameter $\tau$ controls the strength of the adversarial attack, similar to the classical adversarial budget $\varepsilon$. That means a larger value for $\tau$ plays the role of a larger budget and results in a weaker regularization strength. Note that $R(f, \tilde \mu)$ is bounded -- from below by zero and from above by one -- and we do not lose any generality by restricting $\tilde \mu$ to the ball induced by $W_2(\tilde \mu_i,\mu_i) \leq \tau$. That ensures the inner supremum is always attained, as it follows from the following Lemma.

\begin{lma}
    Let $(\Xc,||\cdot||)$ be a separable Banach space. Then the $W_2$-ball of radius $\tau$ around $\mu$
    \[\{\tilde \mu \in \mathcal{P}_2(\Xc) \colon W_2(\mu,\tilde \mu) \leq \tau \}\]
    is tight in the space of probability measures with finite second moment $\mathcal{P}_2(\Xc)$.
\end{lma}
\begin{proof}
    The Wasserstein distance is a metric on $\mathcal{P}_2(\Xc)$. Hence, the reverse triangle inequality holds true, and for any measure $\tilde \mu$ in the $W_2$-ball,
    \[\tau \geq W_2(\mu,\tilde \mu) \geq |W_2(\mu,\delta_0) - W_2(\tilde \mu,\delta_0)| = \left| \sqrt{\int_{\Xc} ||x||^2\ d\mu(x)} - \sqrt{\int_{\Xc} ||x||^2\ d\tilde\mu(x)} \right|.\]
    The second moment of $\mu$ is finite, and hence, all the second moments of elements of the set are uniformly bounded. 
    That implies tightness by Markov's inequality. 
\end{proof}

In this setting, the cost function $c_A$ in the problem \eqref{eq:MMOT-like} becomes
\begin{equation*}
    c_{A}((x_1,\dots,x_m)) = \frac{1}{\tau^2} \inf_{x \in \Xc} \sum_{i=1}^m d(x_i,x)^2 = \frac{1}{\tau^2} \sum_{i=1}^m d(x_i,\overline{x})^2,
\end{equation*}
where $\Bar{x}$ is the Fr\'echet mean on $(\Xc,d)$; in the Euclidean space $\Bar{x} = \frac{1}{m}\sum_{i=1}^m x_i$.

The classical adversarial risk \eqref{eq:probadvrisk} is for historical reasons coming from the data perspective and allowing small perturbations of the data points in the metric of the space $\Xc$, leading to the bound on the deviation in $W_\infty$.
In optimal transport, the $W_2$ distance proved to be an excellent measure of deviations of probability distributions and is used, for example, for Wasserstein barycenter \cite{agueh2011barycenters}. 
The statistical relevance of the 2-Wasserstein distance as a measure of deviation in the space of probability measure is also well understood \cite{panaretos2019statistical}, motivating it as a reasonable alternative to the classical one.
The idea to consider a different metric in the space of probability measures to define the ambiguous set for the attacker is in line with many works in the field of distributional robust optimization where more general problems are considered (see, e.g., \cite{esfahani2018data} and references therein).

The $W_\infty$-ball puts an upper bound on any essential mass dislocation but does not account for the amount of mass moved below that distance. In contrast, the $W_2^2$ penalty term penalizes any mass dislocation scaling quadratic with the distance in the underlying space $\Xc$.

The second motivation for the $W_2^2$ penalty is a purely algorithmic consideration, as explained in the next subsection.

\subsection{Algorithm}
All feasible configurations now have a finite cost but depend on their deviation from their mean. 
In the linear program \eqref{eq:LP} the cost coefficients become
\[c(r) = 1 + \frac{1}{\tau^2} \sum_{i=1}^m d(x_{r_i},\Bar{x})^2.\]

An admissible solution is still given by all singleton configurations; \eqref{eq:RP} with $\Omega = \{\{1\},\dots,\{N\}\}$, but we now cannot exclude a configuration based on whether the cost coefficient is finite or not. 
First, that is an obstacle for exhaustive search. We need the critic as in classical column generation for that decision. 

The corresponding dual program for \eqref{eq:RP} is
\begin{align*}
    \operatorname*{maximize}_{u:\mathbb{R}^N \to \mathbb{R}}\quad &\sum_{i = 1}^N  \frac{u_i}{N} \\
    \text{subject to } &\sum_{i=1}^m u_{r_i} \leq c(r) \text{ for all } r \in \Omega.
\end{align*}
Note that the reduced set of configurations $\Omega$ corresponds to the constraints of the dual problem. The idea of column generation is now that a candidate configuration $r' \notin \Omega$ yields a gain if it violates the constraints given the current optimal dual solution $u^\star$:
\[\sum_{i=1}^{m'} u_{r_i} > c(r').\]
Only those candidates are then added to $\Omega$.

However, with this rule, the number of configurations still increases, and the routine does not exploit the guaranteed sparsity of the optimizers discussed at the beginning.
The second idea from GenCol now is that unused configurations, i.e., $r \in \Omega$ such that $\gamma(r) = 0$, are removed. 
Therefore, we introduce a parameter $\beta$ to limit the number of configurations in $\Omega$. Since the number of active configurations ($\gamma(r) > 0$) of an extremal solution of \eqref{eq:RP} is bounded by the number of constraints ($N$), whenever the number of configurations in $\Omega$ exceeds $\beta\cdot N$ we remove a batch of unused configurations from $\Omega$.
In the following simulations, $\beta$ is a small integer (chosen to be 3 in the numerical simulations in Section \ref{sec:experiments}), and the number of removed configurations is simply $N$.
The routine is summarized in Algorithm \ref{alg:w2algo}. 

\begin{algorithm}[!h]
\caption{Genetic Column Generation for OT-regularized problems}
\label{alg:w2algo}
\begin{algorithmic}[1]
\Require $\Omega = \{\{i\} \colon i \in N\}$
\While{new configurations are found}
    \State Solve \eqref{eq:RP}
    \State $(\gamma^\star, u^\star) \leftarrow$ primal and dual solution to \eqref{eq:RP}
    \If{$|\Omega| > \beta\cdot N$}
        \State remove $N$ inactive configurations from $\Omega$
    \EndIf
    \For{$s$ samples}
        \State Draw parent $r \in \operatorname{spt}(\gamma^\star)$
        \State Generate offspring $r'$ from the parent
        \State gain := $\sum_{i=1}^{m'} u_{r_i} - c(r')$
        \If{gain $> 0$}
            \State Add offspring to $\Omega$
        \EndIf
    \EndFor
\EndWhile
\State\Return solution of \eqref{eq:RP}
\end{algorithmic}
\end{algorithm}

\textbf{Remark:} The penalty term $W_2^2(\tilde \mu, \mu)$ is strictly positive whenever $\tilde \mu \neq \mu$. That implies that the optimal value of \eqref{eq:MMOT-like} is the optimal value of the regularized problem $R - C$. To compute the risk $R$ of the adversarial attack $\tilde \mu$, one has to correct the value by $C$, or, equivalently, replace the cost coefficients of the optimizer $\gamma$ in the linear program formulation \eqref{eq:RP} by 1.

\section{Experiments}
\label{sec:experiments}
In this section, we compare the three proposed strategies. For explanatory purposes, we start with a synthetic data set in $\R^2$ before considering real-world image data. 

The exhaustive search (Algorithm \ref{alg:exhaust}) can directly be compared to the genetic search (Algorithm \ref{alg:gensearch}). The $W_2$ regularization has a different effect and must be considered separately. Therefore, there are two things we want to find out:
\begin{enumerate}
    \item Is the genetic search rule (Algorithm \ref{alg:gensearch}) able to find an optimal set of configurations compared to the exhaustive search (Algorithm \ref{alg:exhaust})?
    \item How does the $W_2$-regularized problem behave in terms of convergence of Algorithm \ref{alg:w2algo} and regularization strength $\tau$?
\end{enumerate}

\subsection{Data}
\textbf{Synthetic data.} We consider ten two-dimensional normal distributions with slightly shifted centers. The number of sampled data points is $N=1000$. The classes and centers were drawn at random, resulting in slightly different sizes for each class. The largest class contains 119 points, and the smallest is 81. The data set is visualized in Figure \ref{fig:TwoDim-data}.
\begin{figure}
    \centering
    \includegraphics[width=0.45\textwidth]{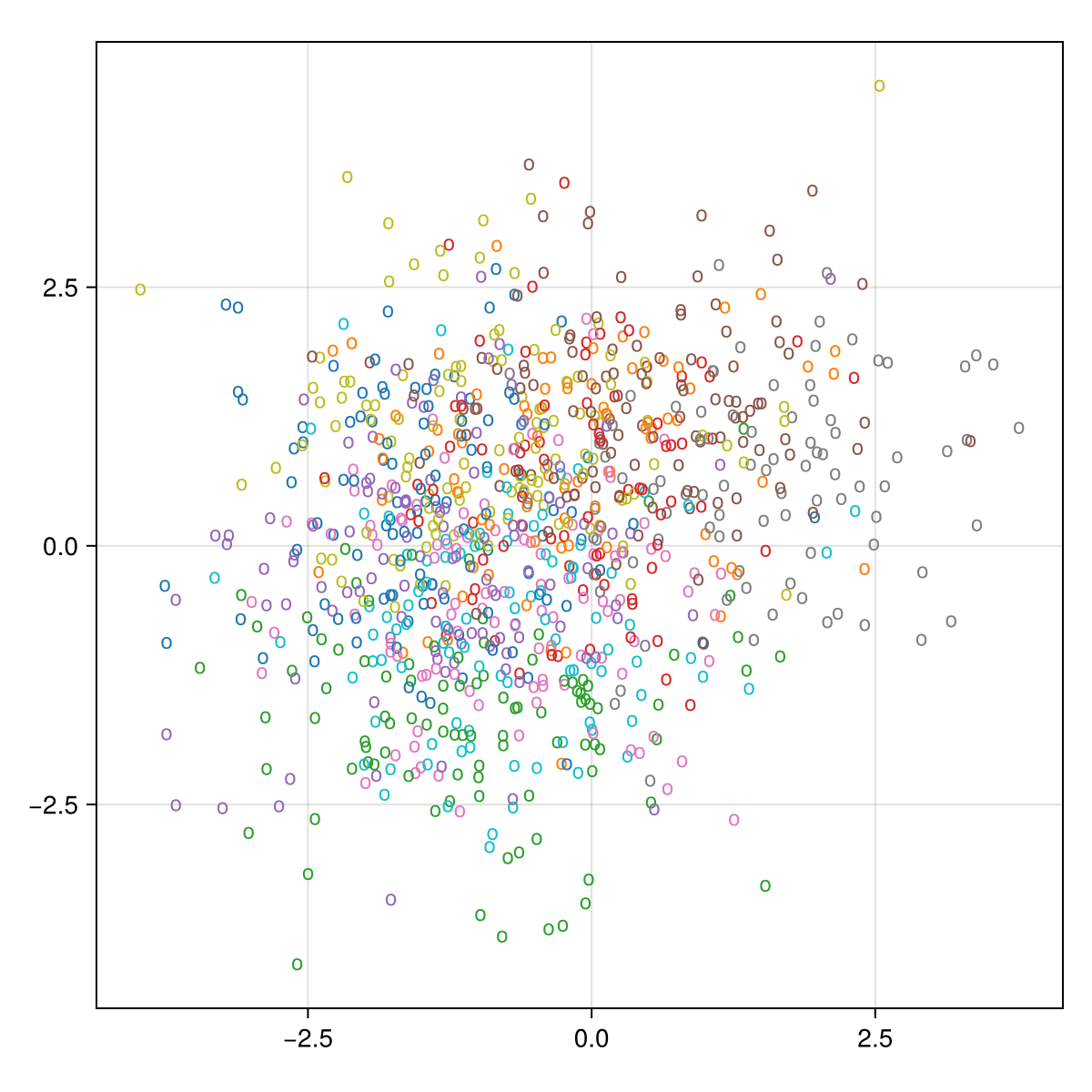}
    \caption{A synthetic data set of 10 overlapping Gaussian distributions. The proximity of the clusters limits the classification power and makes it a hard classification problem.}
    \label{fig:TwoDim-data}
\end{figure}

On purpose, we have a huge overlap of classes, limiting the classification power of any classifier on the underlying distributions and making the classification problem harder. Note that a 1-nearest neighbor classifier has risk 0 on the data since no two data points are identical.\\[1em]
\noindent\textbf{CIFAR-100} is a well-known benchmark data set \cite{krizhevsky2009learning}, publicly available in the internet. It consists of 60000 tiny images of resolution $32\times 32$ pixels with 3 channels (RGB) in 100 classes.

In contrast to the MNIST data set of handwritten images, the larger number of classes and the smaller number of images per class make it much harder to compute the adversarial risk since the number of possible configurations scales exponentially with the number of classes.

The quantity of interest should be compared to the risk of a classifier on the test set. Hence, we used the test split ($N=10000$) for the analysis. The computation time for the cost coefficients, being the radius of the smallest enclosing ball, takes significantly longer for data points in $\R^{3\cdot32^2}$ instead of $\R^2$. For the simulations below, the data set was restricted to the first 30 classes, resulting in $N=3000$ data points to explore a larger range of budgets. 

\subsection{Simulations}
The code for the simulations was written in Julia. The CIFAR-100 data set was downloaded from the official website using the Julia package \href{https://juliaml.github.io/MLDatasets.jl/v0.7.14/}{MLDataSets.jl}. The smallest enclosing ball for the Euclidean metric was computed using the Julia package \href{https://github.com/JuliaFEM/BoundingSphere.jl}{BoundingSphere.jl}.
For solving the linear programs, we used the \href{https://highs.dev/}{HiGHS} optimizer via the Julia package \href{https://jump.dev/JuMP.jl/v1.20.0/}{JuMP.jl}. This allowed an efficient framework to modify and resolve the reduced problems. An additional benefit is that the solver can be exchanged easily: For the large problems resulting from the exhaustive search, we used Mosek \cite{mosek}, yielding a significant speed up.

For better stability in the subsequent linear program solvers, the problems were rescaled in that the mass of each marginal point was set to 1 instead of $1/N$, resulting in a total mass of $\sum_{r\in \Omega} \gamma(r) = N$ instead of 1.

As a stopping criterion for the genetic search rules, a maximum time was chosen. For the genetic search for the classical problem in Algorithm \ref{alg:gensearch}, we also stopped when we reached the true optimizer found by an exhaustive search because the algorithms are strict descend algorithms, and no further improvement can occur.
For Algorithm \ref{alg:w2algo}, no optimizer to the full problem can easily be computed, and therefore, we only stopped after the maximum time.

\subsubsection{Synthetic data}
We start with the classical problem with $W_\infty$-regularization, the underlying distance being the Euclidean distance. For the small synthetic data set, we can explore a large range of budgets, namely budgets from $\varepsilon=0$ to 0.28, using the exhaustive search. 

A detailed breakdown by configuration length is illustrated\footnote{Inspired by \cite{trillos2024optimal}.} in Figure \ref{fig:TwoDim-nconf}. The number of configurations of length 1 is always 1000 since all singleton configurations have radius 0 and are, hence, always feasible. The time spent on the search was measured in 3 independent runs per budget using a parallelized code on a 4-core Intel i5 (2.00 GHz). This is presented in Table \ref{tab:TwoDim-nconf}. Note that for the largest budget ($\varepsilon = 0.28$), there are nearly 10 million configurations found.

Linear programs of that size start to become challenging for solvers like HiGHS. We therefore switched to the commercial LP solver from Mosek \cite{mosek} to reduce the solve time.

\begin{figure}[!ht]
    \centering
    \includegraphics[width=0.6\textwidth]{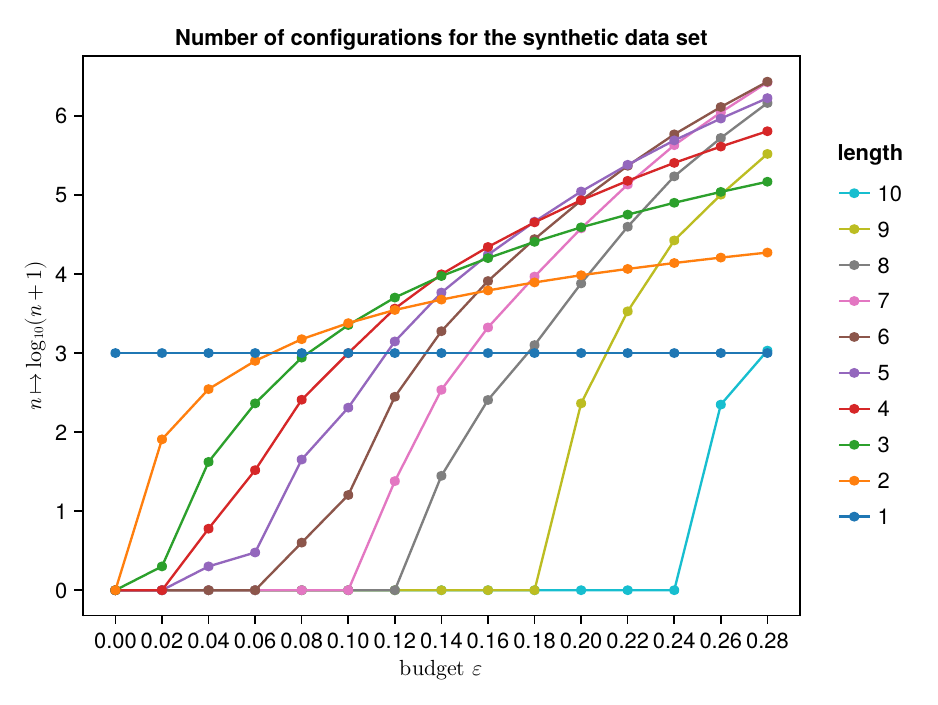}
    \caption{The figure shows the number of configurations per adversarial budget in log-scale. Each colored line indicates configurations of a certain length. The data set features 10 classes, limiting the maximal length of a configuration. The corresponding search times are reported in Table \ref{tab:TwoDim-nconf}.}
    \label{fig:TwoDim-nconf}
\end{figure}

Next, we test the genetic search, as described in Algorithm \ref{alg:gensearch}. We simply choose the weighting of the search rules to be 1:1:0 (i.e. points in a configuration never die). 
The stopping criterion of each routine was either reaching the global optimum as determined by the exhaustive search or running for at most 300 seconds.
The convergence plots for a selection of budgets are presented in Figure \ref{fig:TwoDim-conv}.
For all budgets, the genetic search rule finds a good approximation within a short time. 
\begin{figure}
    \centering
    \includegraphics[width=0.6\textwidth]{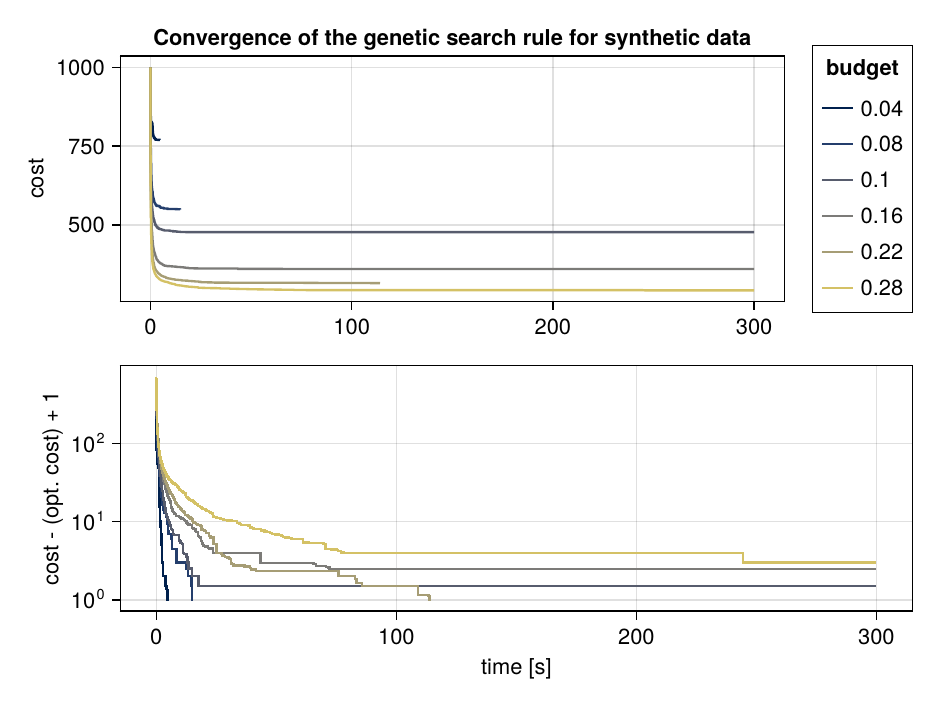}
    \caption{Convergence of the genetic search rule for the classical $W_\infty$-regularization. For small budgets, an optimal set of configurations is quickly found. The upper figure shows the optimal cost of the reduced problem in dependence on computation time; the lower figure shows the difference to the true optimal cost found by the exhaustive search in logarithmic scale. For larger budgets, the optimal cost often stagnates, but as seen for budget $\varepsilon = 0.22$ it is possible to find an optimal set of configurations.}
    \label{fig:TwoDim-conv}
\end{figure}

The resulting estimations for the adversarial risks are shown in Figure \ref{fig:TwoDim-exhvsgen}; the relative error is below 1\% for all budgets, indicating a good approximation by the genetic search rules. 

\begin{figure}
    \centering
    \includegraphics[width=0.6\textwidth]{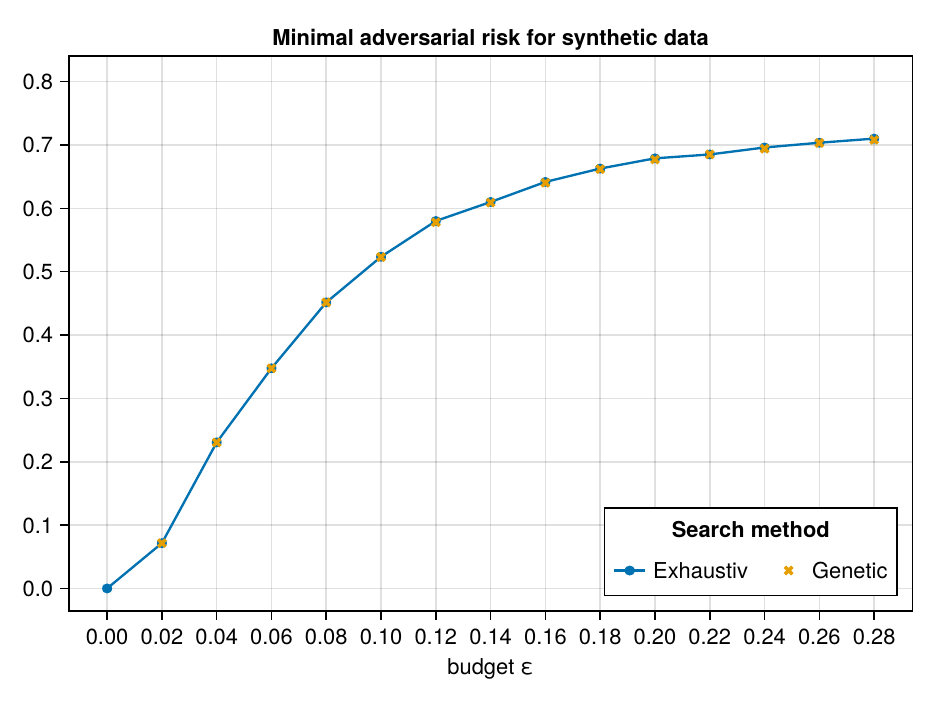}
    \caption{Minimal adversarial risk for the classification problem for the Euclidean metric. The x-axis indicates the adversarial budget $\varepsilon$, and the y-axis indicates the corresponding minimal adversarial risk. The relative error of the genetic search rule compared to the exhaustive search is below 1\%.}
    \label{fig:TwoDim-exhvsgen}
\end{figure}

The restriction of the search space to the active configurations first accelerates the search for configurations, as seen in the fast decay in the first seconds, independent of the budget.

For the $W_2$-regularized problem, we take advantage of the optimal dual solution and use column generation as described in Routine \ref{alg:w2algo}. New configurations are only added if they have a positive gain with respect to the current reduced problem. 

The convergence behavior is shown in Figure \ref{fig:TwoDim-w2} on the left. For small regularization strength $\tau$, the routine converges quickly; for larger $\tau$, we stopped the routine after 300 seconds. The maximal risk for a given $\tau$ might hence be underestimated. 

\begin{figure}
    \centering
    \includegraphics[width=0.45\textwidth]{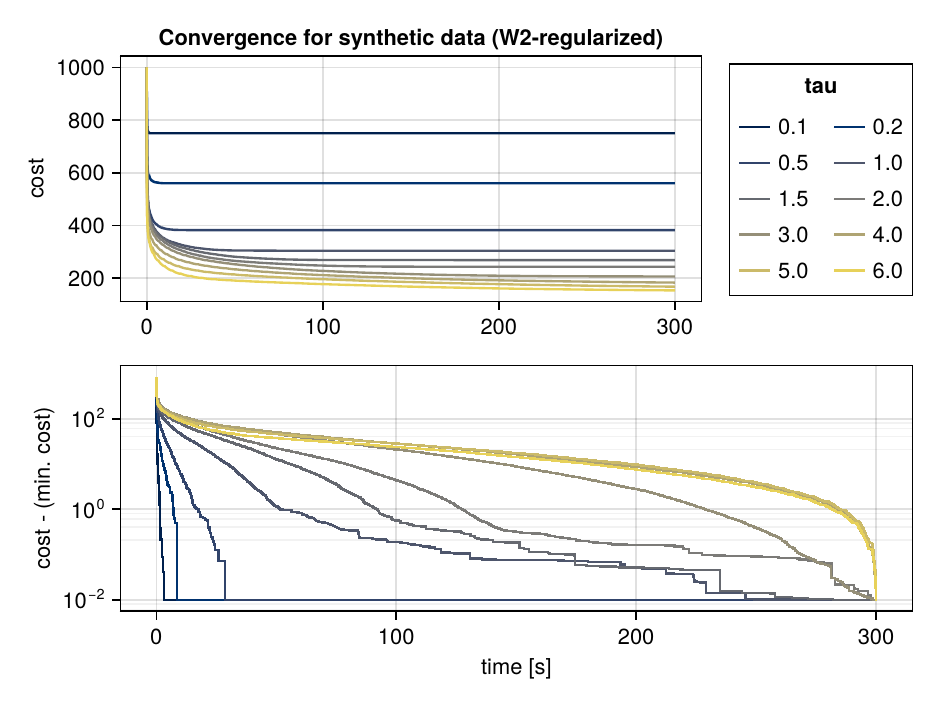}
    \includegraphics[width=0.45\textwidth]{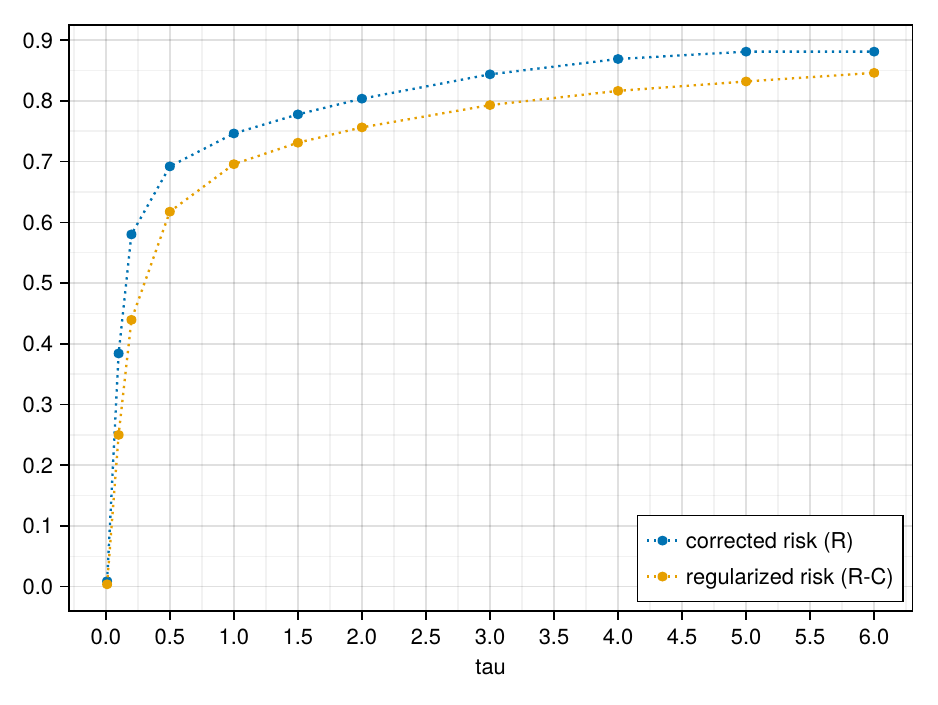}
    \caption{Left: Convergence of Algorithm \ref{alg:w2algo} for the synthetic data set. The convergence speed decreases with increasing $\tau$. The lower plot shows the relative optimal cost of the reduced problem on a logarithmic scale with base 10 to visualize the convergence speed. Right: The optimal values of Problem \ref{eq:w2penalized} for different regularization strengths $\tau \in [0,6]$. For $\tau \geq 5.0$ the corrected adversarial risk is the maximal adversarial risk ($1-\frac{119}{1000}$) for that data set.} 
    \label{fig:TwoDim-w2}
\end{figure}
For $W_2$ penalized problems, the optimal value of the regularized problem does not coincide with the one from the unregularized. The reason is that in contrast to $W_{\infty}$ penalty each deviation $\tilde \mu$ from $\mu$ has a positive cost $W_2^2(\tilde \mu,\mu)$.
Hence, in order to obtain the adversarial risk, we need to correct the optimal value by the value of the penalty term.
The corrected adversarial risks depending on the "budget" $\tau$ are presented in Figure \ref{fig:TwoDim-w2} on the right.

One can see that even if not fully converged, the corrected risk for $\tau \geq 5.0$ is optimal because it is the maximal adversarial risk for the data set. This is simply given by the fraction of the largest class on the size of the data set. That implies that even if additional configurations are found that increase the regularized objective $R - C$, the correction does not affect the estimation of the risk $R$.

We conclude that genetic column generation can be used to compute the minimal adversarial risk, if $W_2$ regularized, for any budget $\tau$ within a reasonable time for this data set.

\subsubsection{CIFAR data set}

Finally, we want to test the algorithms on real-world data. The number of configurations per length and budget found by the exhaustive search is shown in Figure \ref{fig:CIFAR30-exhvsgen} on the left. For an adversarial budget of 5.4 the longest configurations are of length 11. The minimal adversarial risk in dependence on the budget is again well approximated by the genetic search. 

\begin{figure}
    \centering
    \includegraphics[width=0.45\textwidth]{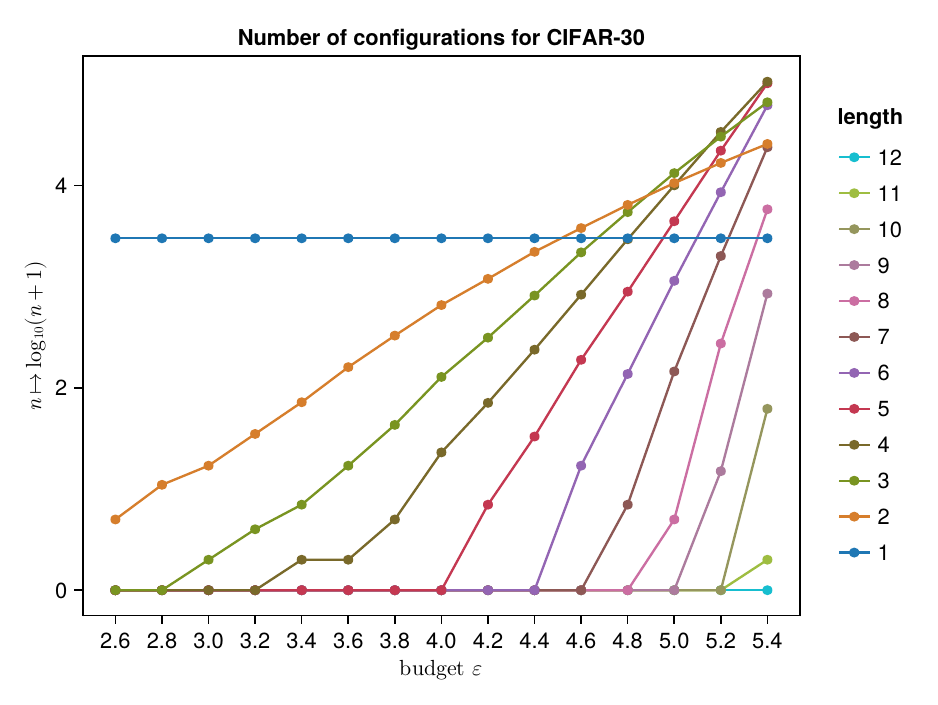}
    \includegraphics[width=0.45\textwidth]{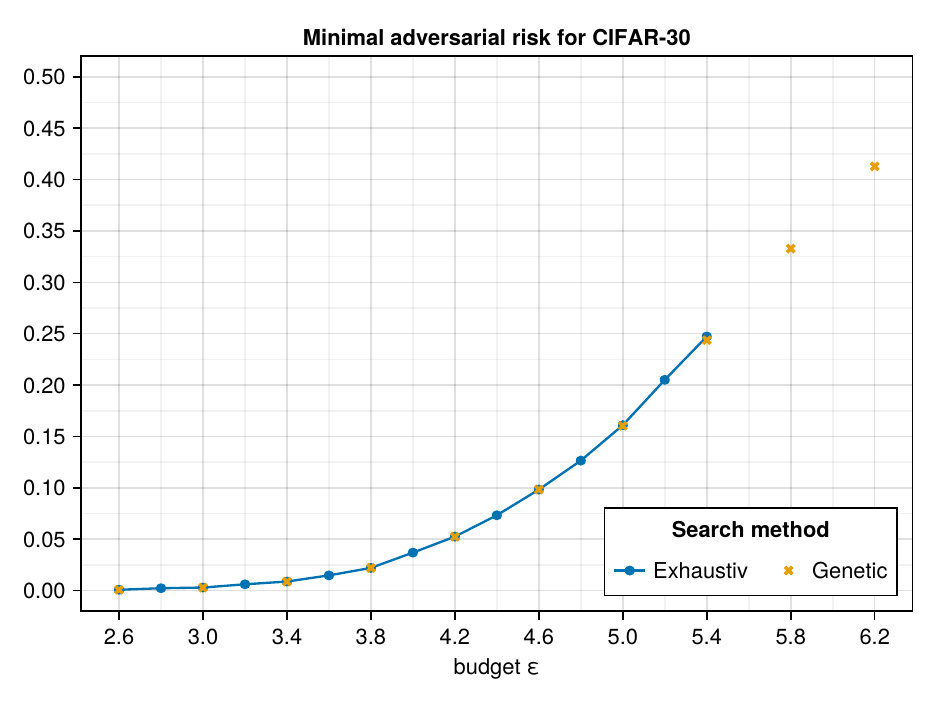}
    \caption{Left: Number of configurations with radius less than the budget (horizontal axis), split by configuration length. Right: Minimal adversarial risk for the classification problem for the Euclidean metric. The x-axis indicates the adversarial budget $\varepsilon$, and the y-axis indicates the corresponding minimal adversarial risk. The blue curve indicates the true results due to the exhaustive search. The yellow dots indicate the estimation due to the genetic search rules. The relative error of the genetic search rule compared to the exhaustive search is always below 1\%.}
    \label{fig:CIFAR30-exhvsgen}
\end{figure}

Even if the number of configurations is not bigger than in the first example, the search took more time because the computation of the cost coefficients -- which is the radius of the configuration --, took significantly more time due to the very high dimensional feature space.
The exhaustive search for the budget $\varepsilon = 5.4$ took about 45 minutes and wasn't carried out for larger budgets. 
For those budgets, the genetic search rule found a good approximation of the minimal adversarial risk, as seen in Figure \ref{fig:CIFAR30-exhvsgen} on the right.
For larger budgets, the genetic search rule can still be used to estimate a lower bound for the minimal adversarial risk, but the algorithm didn't converge out. That indicates that for problems of that size a purely generative genetic search rule is not sufficient.

In contrast, the $W_2$ regularized problem has the advantage that the Algorithm has a powerful critic to accept new configurations. That accelerates convergence significantly. For the $W_2$ regularized problem, the convergence for $\tau \in \{6,7,8,9,10\}$ is presented in Figure \ref{fig:CIFAR30-W2} on the left. 
But again, for even larger $\tau$, the convergence gets significantly slower. The right shows the estimation for the regularized and for the corrected adversarial risk for regularization strength $\tau$. The algorithm converged only for $\tau \leq 7$, implying that the minimal adversarial risk for $\tau > 7$ is underestimated.

\begin{figure}
    \centering
    \includegraphics[width = 0.45\textwidth]{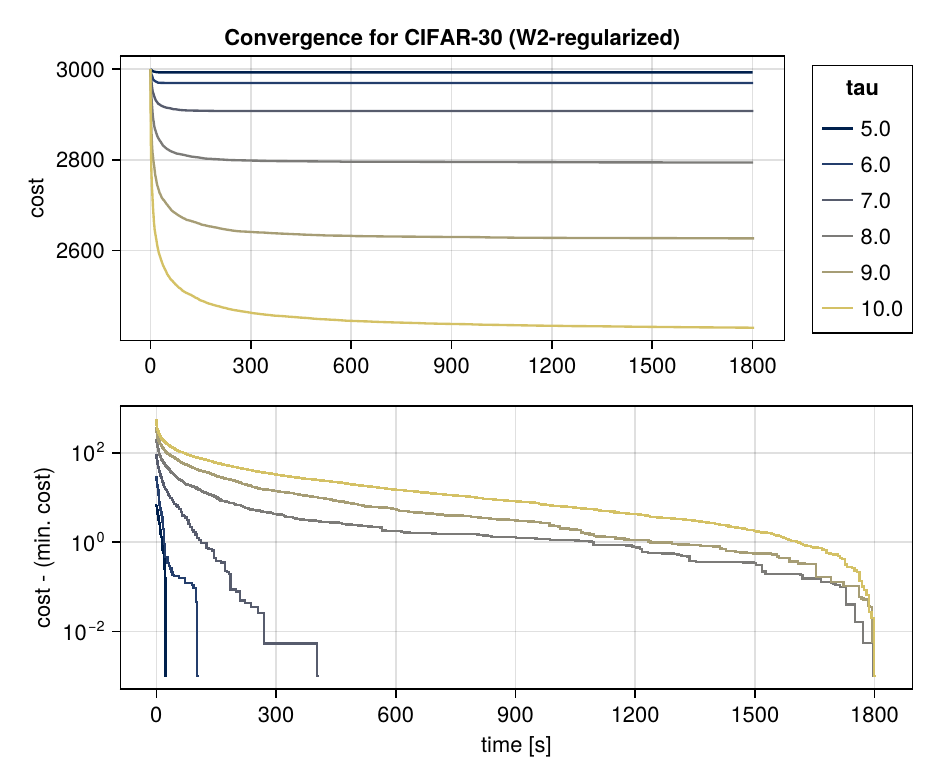}
    \includegraphics[width=0.45\textwidth]{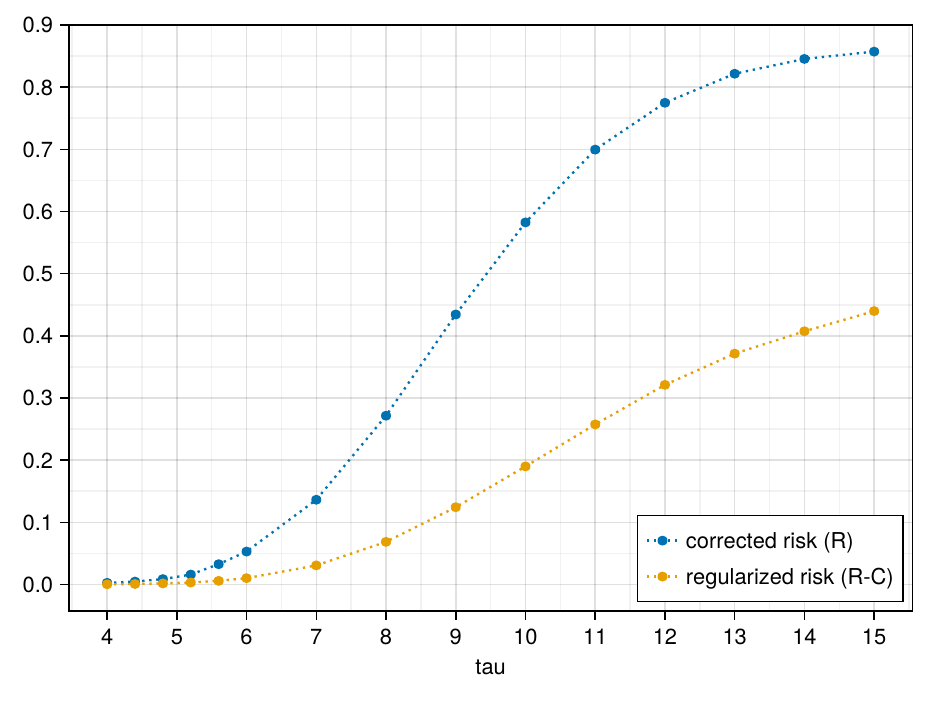}
    \caption{Left: Convergence of Algorithm \ref{alg:w2algo} for the CIFAR data. For small budgets ($\tau \leq 7$) the algorithm converged in less than 400 seconds. For budgets $\tau \leq 10$, the algorithm still asymptotically found a good approximation indicated by a low discrete gradient. Right: Lower bounds for the minimal adversarial risk. For large $\tau > 10$, the minimal adversarial risk is (still) underestimated.}
    \label{fig:CIFAR30-W2}
\end{figure}

\section{Conclusion}
We investigated how ideas from Genetic Column Generation can be used to find the minimal adversarial risk for multi-class classification problems, especially for data sets with many overlapping classes. We further explored the option to replace the classical adversarial budget with respect to a $W_\infty$ ball by a penalty on the $W_2$ deviation.

By restricting the set of configurations and solving the reduced problem, we ensured finding a lower bound for the minimal adversarial risk and -- for budgets not too big -- a very good approximation of it.
We saw that a genetic search rule alone can be used to quickly find an approximation from below but might not find the minimal adversarial risk. For small to moderate budgets, that approximation was still good, considering significantly fewer configurations.

By replacing the classical adversarial attack with a $W_2$ penalty, we were able to explore a slightly different problem. The accelerated convergence behavior by utilizing duality enabled us to explore a large range of penalty strengths up to regimes of much larger adversarial risk.

In both algorithms, the curse of dimension occurring in the number of configurations to be considered was efficiently tackled by considering an iterative sequence of reduced problems and updating the set of configurations in a genetic fashion. In the $W_2$-regularized problem, the restriction of the problem size did not harm the algorithm to efficiently find new configurations.

However, a few open question remain. First, the genetic scheme is quite flexible and many other proposal rules could be tried. Second, one might gain some computational advantages by parallelizing the search for new configurations and using larger computers. And finally, the optimal dual solution can be used to define a classifier. It would be interesting to compare it with existing classifiers.

\bibliographystyle{alpha}
\bibliography{lit} 
\clearpage
\appendix
\section{Computation times for exhaustive search}

\begin{table}[!ht]
    \centering
    \begin{tabular}[b]{l|r|r}
    \hline
        budget & no. configs & mean search time (sd) \\ \hline
        0.08 & 3676 & 0.08 (0.013) \\ 
        0.1 & 6870 & 0.16 (0.065) \\ 
        0.12 & 14905 & 0.19 (0.057) \\
        0.14 & 33191 & 0.39 (0.025) \\ 
        0.16 & 73130 & 0.78 (0.017) \\
        0.18 & 163638 & 1.64 (0.065) \\ 
        0.2 & 377204 & 3.93 (0.248) \\ 
        0.22 & 874413 & 9.22 (0.369) \\ 
        0.24 & 2051231 & 22.54 (0.441) \\
        0.26 & 4500911 & 49.41 (0.775) \\ 
        0.28 & 9657249 & 106.77 (1.044) \\
    \end{tabular}
    \caption{The computation times for the exhaustive search on the synthetic data example. The time was measured in 3 independent runs using a parallelized code on 4 cores. The number of feasible configurations with finite cost quickly blows up, as does the time spent.}
    \label{tab:TwoDim-nconf}
\end{table}

\end{document}